\documentclass[12pt]{amsart}
\usepackage{amsmath}
\usepackage{amsthm}
\usepackage{amsfonts}
\usepackage{amssymb}
\usepackage{amscd}

\usepackage{rotating}

\usepackage{graphicx}

\theoremstyle{plain}
\newtheorem{thm}{Theorem}
\newtheorem{lemma}[thm]{Lemma}

\newtheorem*{conj}{Infinite Log-concavity Conjecture}

\theoremstyle{definition}
\newtheorem{definition}[thm]{Definition}
\newtheorem{example}[thm]{Example}

\theoremstyle{remark}

\newtheorem*{question}{Question}
\newtheorem*{conjec}{Conjecture}

\newcommand{\Z}{\ensuremath \mathbb{Z}}

\newcommand{\calP}{\mathcal{P}}

\DeclareMathOperator{\sgn}{sgn}

\begin{document}

\title[Nonnegative minors of minor matrices]
{Nonnegative minors of minor matrices}

\author{David A. Cardon}
\address{Department of Mathematics, Brigham Young University, Provo, UT 84602}
\email{cardon@math.byu.edu}

\author{Pace P. Nielsen}
\address{Department of Mathematics, Brigham Young University, Provo, UT 84602}
\email{pace@math.byu.edu}


\begin{abstract}
Using the relationship between totally nonnegative matrices and directed acyclic weighted planar networks, we show that $2\times 2$ minors of minor matrices of totally nonnegative matrices are also nonnegative.  We give a combinatorial interpretation for the minors of minor matrices in terms of the weights of families of paths in a network.
\end{abstract}
\maketitle

\section*{Introduction}

By attaching weights to the edges of a finite, directed, acyclic planar network we form the corresponding weight matrix.  This weight matrix encodes important information about the network.  For the types of networks relevant to this paper, a result of Lindstr\"om \cite[Lemma 1]{Lindstrom} shows that these matrices are totally nonnegative, i.e.\ any minor is a subtraction-free expression in the weights of the network.  In this paper we extend Lindstr\"om's argument by showing that $2\times 2$ minors of the minor matrices (defined in \S\ref{sec:weightmatrices}) of the weight matrix are also nonnegative.  Moreover, we show that these minors of the minor matrices will be subtraction-free expressions in the weights of the original network.

As an application of the main theorem of this paper we give an extension of a conjecture, independently made by McNamara and Sagan \cite[Conjecture 7.1]{MS} and R.\ P.\ Stanley, about infinite log-concavity.  To state their conjecture we introduce some of the relevant background.  Let $\{a_{n}\}_{n=0}^{\infty}$ be a sequence of nonnegative real numbers.  We say the sequence is \emph{log-concave} if the new sequence $\{b_{n}\}$ given by $b_{n}=a_{n}^{2}-a_{n-1}a_{n+1}$ still consists of nonnegative numbers, where $a_{-1}=0$.  If every iteration of this procedure creates another nonnegative sequence, then we say that the original sequence is \emph{infinitely log-concave}.  Notice that if a polynomial $\sum_{i=0}^{m}a_{i}x^{i}$ has only real negative roots, then the sequence $\{a_{n}\}_{n=0}^{\infty}$ (where $a_{n}=0$ if $n>m$) is nonnegative.  The statement is as follows:

\begin{conj}
If $\sum_{i=0}^{m}a_{i}x^{i}$ has only real negative roots then the polynomial $\sum_{i=0}^{n}(a_{i}^{2}-a_{i-1}a_{i+1})x^{i}$ also has only real negative roots.  In particular, the sequence $\{a_{n}\}$ is infinitely log-concave.
\end{conj}
\noindent Petter Br\"{a}nd\'{e}n \cite{Branden} recently proved this conjecture, using complex-analytic techniques applied to symmetric polynomials.  We were led to our extension (which is stated in \S\ref{Section:Open}) by first noticing that the sequence $\{a_{n}\}$ gives rise to a totally nonnegative matrix $A$ and the infinite log-concavity conjecture would follow from the total nonnegativity of a certain matrix (which we call a minor matrix) formed from $A$ by taking successive minors.

\section{Planar Networks, Weight Matrices, and Minor Matrices}
\label{sec:weightmatrices}

A fundamental object of this paper is a special type of planar network called a planar network of order $n$, which we define below.  To this network is associated an $n \times n$ matrix called the weight matrix. In Theorem~\ref{thm:MainTheorem} we will show that certain matrices derived from the weight matrix, which we call minor matrices, satisfy an important nonnegativity property.

\begin{definition} \label{def:planarnetwork}
A \textit{planar network of order $n$} is a finite directed acyclic planar graph containing exactly $n$ sources and $n$ sinks, denoted  $s_1,\ldots,s_n$ and $t_1,\ldots,t_n$ respectively, which lie on the boundary. Furthermore, the sources and sinks are configured such that they may be  labeled in counterclockwise order as $s_1,\ldots,s_n,t_n,\ldots,t_1$. It will be assumed that the network is drawn with the sources $s_1,\ldots,s_n$ on the left and the sinks $t_1,\ldots,t_n$ on the right, with no vertical edges, and with the edges directed from left to right.  An example is given in Figure~\ref{fig:PlanarNetworkExample01}.  A non-example is given in Figure~\ref{fig:PlanarNetworkNotOrderN}; the planar network in that figure is not of order $n$ for any $n\geq 1$, because the sources and sinks cannot be ordered in the appropriate manner.
\end{definition}

\begin{figure}[ht]
\includegraphics{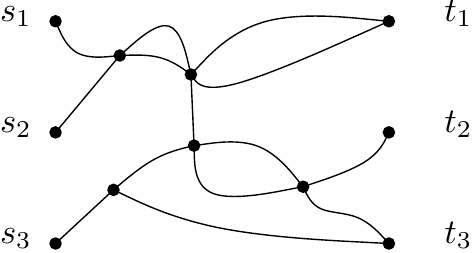}
\caption{An example of a planar network of order 3.  All edges are directed to the right.}
\label{fig:PlanarNetworkExample01}
\end{figure}

\begin{figure}[ht]
\includegraphics{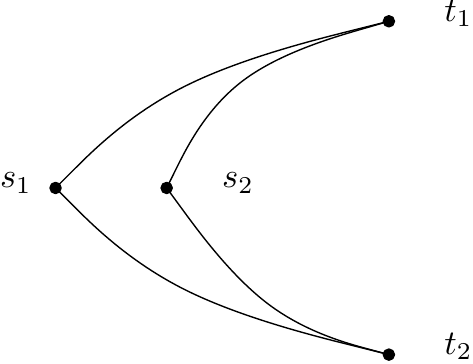}
\caption{An example of a directed, acyclic planar network, with an equal number of sources and sinks, which is not of order 2.}
\label{fig:PlanarNetworkNotOrderN}
\end{figure}

Given a planar network $\Gamma$ of order $n$ we assign indeterminates to each of the edges, which we think of as weights.  In applications, we may specialize these weights to be real numbers. An example of a planar network of order $3$, with weights, is given in Figure~\ref{fig:Weights}.

\begin{figure}[ht]
\includegraphics{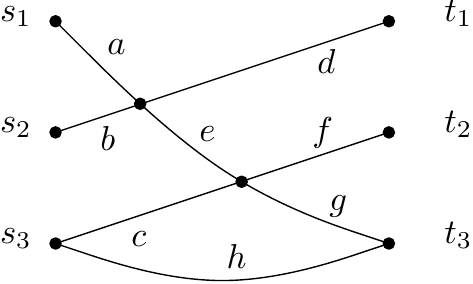}
\caption{A planar network of order $3$ with weights.}
\label{fig:Weights}
\end{figure}
By a \textit{path} $\pi$ in $\Gamma$ we mean a directed continuous curve in the network beginning at a source $s_i$ and terminating at a sink $t_j$.  A family of paths is \textit{vertex-disjoint} if no two paths from the family intersect.  The \textit{weight} of $\pi$, denoted $\omega(\pi)$, is the product of the weights of the edges of $\pi$.  For example, in Figure~\ref{fig:Weights}, there is only one path from $s_{1}$ to $t_{2}$, and it has weight $aef$.

\begin{definition} \label{def:weightmatrix}
The \textit{weight matrix} $W=W(\Gamma)$ of a planar network $\Gamma$ of order $n$ is the $n \times n$ matrix $W=(w_{i,j})$, where
\[
w_{i,j}=\sum_{\pi\in P_{i,j}} \omega(\pi)
\]
and $P_{i,j}$ is the set of paths from source $s_i$ to sink $t_j$.  By convention empty sums are $0$.
\end{definition}

\begin{example}\label{Example:WeightMatrix}
The planar network in Figure~\ref{fig:Weights} has weight matrix
\[
W=
\begin{pmatrix}
ad & aef & aeg \\
bd & bef & beg \\
0 & cf & cg+h
\end{pmatrix}.
\]
\end{example}

We are particularly interested in determinants of submatrices of these weight matrices.  To this end we introduce some notation to simplify the formation of arbitrary minors.  For any positive integer $k\in \Z$, we let $[k]=\{1,2,3,\ldots, k\}$.  Let $W=(w_{i,j})$ be any $m \times n$ matrix, and let $I\subseteq [m]$ and $J\subseteq [n]$ be sets of indices of equal cardinality.  Write
\begin{align*}
I & = \{i_1,\ldots,i_k\}, \text{ where } i_1< i_2 < \cdots < i_k, \text{ and}\\
J & = \{j_1,\ldots,j_k\}, \text{ where } j_1< j_2 < \cdots < j_k.
\end{align*}
Then by $W[I,J]$ we denote the $k \times k$ submatrix
\[
W[I,J] = (w_{i,j}), \qquad (i \in I,\ j \in J),
\]
with rows indexed by $I$ and columns indexed by $J$. The $(I,J)$-minor of $W$ is the determinant
\[
\det W[I,J] = \sum_{\sigma \in S_k} \sgn(\sigma) \prod_{\ell=1}^k w_{i_{\ell},j_{\sigma(\ell)}},
\]
where $S_k$ is the group of permutations of the set $[k]$. Recall that a matrix $W$ is \textit{totally nonnegative} (abbreviated TN) if each of its minors is nonnegative.

A well-known result due to  Lindstr\"om (which we give as Lemma~\ref{lemma:Lindstrom} below) is that the minors of the weight matrix $W$ of a planar network of order $n$ are subtraction-free expressions in terms of the weights of the network. Thus, when the weights are positive real numbers, the weight matrix is totally nonnegative.  For example, by direct computation one can verify that all minors of the
matrix in Example~\ref{Example:WeightMatrix} are subtraction-free expression in terms of the weights $a,b,c,\ldots,g,h$.

There are a number of different generalizations of Lindstr\"om's Lemma; for example, see the section on looped-erased walks in Postnikov~\cite{Postnikov2007}.  The main result of the paper, Theorem~\ref{thm:MainTheorem}, extends Lindstr\"om's Lemma from the weight matrix to another matrix, called the \textit{minor matrix}, whose definition is given below.

\begin{definition}
Let $A$ and $B$ be sets of equal cardinality $k$.  We write them, under the usual ordering of integers, in the form
\begin{align*}
A & = \{a_1,\ldots,a_k\} \subseteq \{0,1,2,\ldots,m-1\}, \text{ and} \\
B & = \{b_1,\ldots,b_k\} \subseteq \{0,1,2,\ldots,n-1\}.
\end{align*}

The \textit{$(A,B)$-minor matrix} $T=(t_{i,j})$ of an $m \times n$ matrix $W$ is the matrix whose entries are defined in terms of minors of $W$ by
\begin{equation}
t_{i,j} = \det W[i+A,j+B],
\end{equation}
where $i+A=\{i+a_1,\ldots,i+a_k\}$ and $j+B = \{j+b_1,\ldots,j+b_k\}$ and where $1 \leq i \leq m-a_k$ and $1 \leq j \leq n - b_k$.
\end{definition}

There is a connection between minor matrices and log-concavity.  Consider the following example:

\begin{example}
Let $A=B=\{0,1\}$ and let $W=(w_{i,j})$ be $n \times n$. The $(A,B)$-minor matrix of $W$ is the $(n-1) \times (n-1)$ matrix $T=(t_{i,j})$ whose entries are consecutive $2 \times 2$ minors of $W$ where
\[
t_{i,j}=\det W[\{i,i+1\},\{j,j+1\}] = w_{i,j}w_{i+1,j+1}-w_{i,j+1}w_{i+1,j}.
\]

In particular, if we are given a sequence $\{a_{m}\}_{m=0}^{n-1}$ of numbers, and we set $w_{i,j}=a_{j-i}$, then the $(A,B)$-minor matrix has entries $t_{i,j}=a_{j-i}^{2}-a_{j-i-1}a_{j-i+1}$.  These are the numbers which arise in the log-concavity definition.  It turns out that the the infinite log-concavity conjecture is equivalent to the assumption that if an expanded form of the matrix $W$ is TN, then the new matrix $T$ is also totally nonnegative.  This connection is spelled out more completely in \S\ref{Section:Open}.
\end{example}

With all of this terminology in place, we can now state the main theorem of the paper:

\begin{thm} \label{thm:MainTheorem}
Let $\Gamma$ be a planar network of order $n$ with weighted edges.  If $T$ is the $(A,B)$-minor matrix of the weight matrix of $\Gamma$, then every $2\times 2$ minor of $T$ is a polynomial in terms of the weights having no negative coefficients. In other words, every $2\times 2$ minor of $T$ is a subtraction-free expression in terms of the weights of $\Gamma$.
\end{thm}

The proof of this theorem is given in \S\ref{sec:Lindstrom} through \S\ref{sec:CompletionOfProof}.  This theorem is sharp, as we will give an example of a planar network of order $6$, for which one minor matrix has a $3\times 3$ minor which is negative.  However, computations suggests that placing extra conditions on $\Gamma$ may be sufficient to force all $(A,B)$-minor matrices to be TN.

\section{A Lemma of Lindstr\"{o}m} \label{sec:Lindstrom}

Lindstr\"om~\cite{Lindstrom} (and earlier, in another context, Karlin and McGregor \cite{KM}) showed that the weight matrix of a planar network is totally nonnegative. Conversely, every TN matrix is the weight matrix of some planar acyclic network with edges having positive real weights, which was first proved by Brenti \cite{Brenti} (see also \cite{Talaska}).  Since the proof of the main theorem in this paper both depends on and generalizes Lindstr\"om's lemma, we include it for the sake of completeness.

\begin{lemma}[Lindstr\"om] \label{lemma:Lindstrom}
The minors of the weight matrix $W$ of a planar network $\Gamma$ of order $n$ are subtraction-free expressions in the weights of the network. If the weights are positive real numbers, the weigh matrix is totally nonnegative. Furthermore, the $(I,J)$-minor of $W$ is equal to the sum of the weights of all vertex-disjoint families of paths from the source points indexed by $I$ to the terminal points indexed by $J$.
\end{lemma}

\begin{proof}
Since the $(I,J)$-minor of the weight matrix $W$ is the determinant of the weight matrix of the subgraph consisting of the paths from the sources indexed by $I$ to the sinks indexed by $J$, it suffices to prove the lemma in the case of the full weight matrix: $I=J=[n]$.

As before, we let $S_n$ denote the group of permutations of the set $[n]=\{1,\ldots,n\}$. Let $\calP_{\sigma}$ denote the set of families $\boldsymbol{\pi}=\boldsymbol{\pi}_{\sigma}=(\pi_1,\ldots,\pi_n)$ where $\pi_i$ is a path from $s_i$ to $t_{\sigma(i)}$. If $\boldsymbol{\pi} \in \calP_{\sigma}$, we will say that $\sgn(\boldsymbol{\pi})=\sgn(\sigma)$. Let $\calP$ be the set of all such path families:
\[
\calP = \bigcup_{\sigma \in S_n} \calP_{\sigma}.
\]
Let $\omega(\boldsymbol{\pi})=\prod_{i=1}^n \omega(\pi_i)$ denote the product of the weights of the paths in the family $\boldsymbol{\pi}$.  Recall that the $(i,j)$-entry of the weight matrix, denoted $w_{i,j}$ or $w(i,j)$, is the sum of the weights of all paths from source $s_i$ to sink $t_j$. Thus,
\begin{equation}   \label{eqn:detW}
\det W
= \sum_{\sigma \in S_n} \sgn(\sigma)\prod_{k=1}^n w(k,\sigma(k))
= \sum_{\sigma \in S_n} \sum_{\boldsymbol{\pi} \in P_\sigma}\sgn(\boldsymbol{\pi}) \omega(\boldsymbol{\pi}).
\end{equation}

We will show that the only non-canceling terms in the determinant correspond to vertex-disjoint path families associated with the identity permutation. Subdivide the set $\calP$ of path families in $\Gamma$ into three disjoint subsets as follows:
\[
\calP= \calP_{0} \cup \calP_{+} \cup \calP_{-},
\]
where
\begin{align*}
\calP_{0} & = \{ \boldsymbol{\pi} \in \calP \, : \, \text{$\boldsymbol{\pi}$ is vertex-disjoint} \}, \\
\calP_{-} & = \{ \boldsymbol{\pi} \in \calP \, : \, \text{$\boldsymbol{\pi}$ is not vertex-disjoint and $\sgn(\boldsymbol{\pi}) = -1$}\}, \\
\calP_{+} & = \{ \boldsymbol{\pi} \in \calP \, : \, \text{$\boldsymbol{\pi}$ is not vertex-disjoint and $\sgn(\boldsymbol{\pi}) = +1$}\}.
\end{align*}
Examples of path families in $\calP_{0}$, $\calP_{+}$, and $\calP_{-}$ are illustrated in Figure~\ref{fig:LindstromDecomposition}; the two paths in the path family are given by a dashed and solid line, respectively.  Note that path families in $\calP_{0}$ necessarily correspond to the identity permutation. Equation~\eqref{eqn:detW} becomes
\begin{equation} \label{eqn:PPP}
\det W =
\sum_{\boldsymbol{\pi} \in \calP_{0}} \omega(\boldsymbol{\pi})
-\sum_{\boldsymbol{\pi} \in \calP_{-}} \omega(\boldsymbol{\pi})
+\sum_{\boldsymbol{\pi} \in \calP_{+}} \omega(\boldsymbol{\pi}).
\end{equation}
We will establish a bijection between $\calP_{-}$ and $\calP_{+}$ that preserves weights. Thus equation~\eqref{eqn:PPP} will reduce to
\[
\det W =
\sum_{\boldsymbol{\pi} \in \calP_{0}} \omega(\boldsymbol{\pi}),
\]
proving the theorem.

\begin{figure}[ht]
\includegraphics{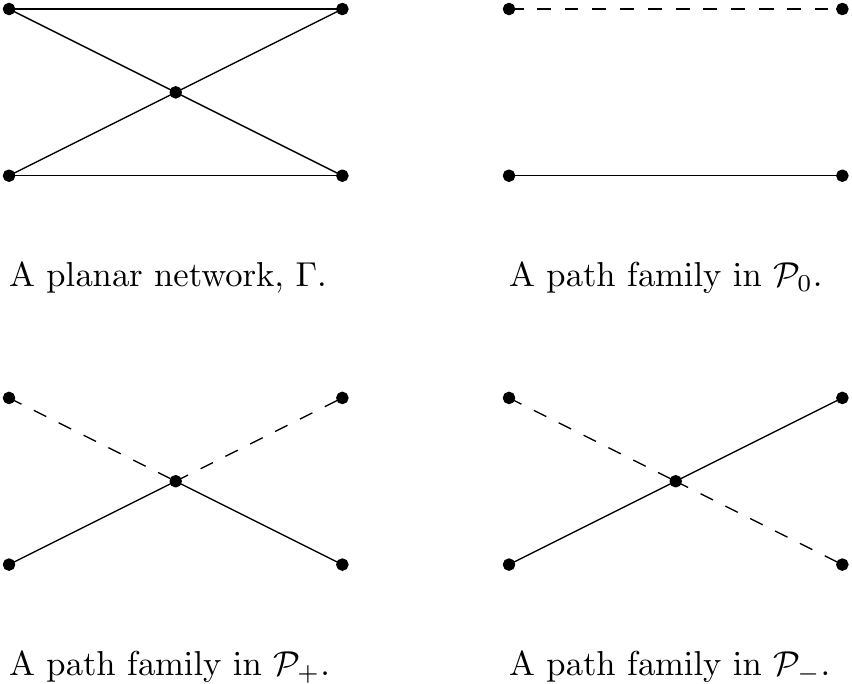}
\caption{Examples of path families in the different subsets of $\mathcal{P}$.}
\label{fig:LindstromDecomposition}
\end{figure}

By slightly perturbing the planar network if necessary, we can guarantee that no two vertices (apart from sources and sinks) lie on the same vertical line.  Let $\boldsymbol{\pi}=(\pi_1,\ldots,\pi_n)$ be a path family in $\calP$ that is not vertex-disjoint. Then there is a rightmost node at which at least two of the paths intersect. Let $i$ and $j$, with $i<j$, be the least two indices of paths $\pi_i$ and $\pi_j$ in $\boldsymbol{\pi}$ that intersect at this node. Form new paths $\pi'_i$ and $\pi'_j$ by interchanging the portions of $\pi_i$ and $\pi_j$ to the right of the rightmost intersection node. This gives a new path family
\[
\boldsymbol{\pi}'=(\pi_1,\ldots,\pi'_i,\ldots,\pi'_j,\ldots,\pi_n)
\]
such that
\[
\omega(\boldsymbol{\pi})=\omega(\boldsymbol{\pi}') \quad \text{and} \quad \sgn(\boldsymbol{\pi}) = - \sgn(\boldsymbol{\pi}').
\]

The mapping
\begin{align*}
\calP_{-} & \rightarrow \calP_{+} \\
\boldsymbol{\pi} & \mapsto \boldsymbol{\pi}'
\end{align*}
is a weight preserving bijection. An example of this path swapping construction and the bijection is illustrated in Figures~\ref{fig:LindstromPathSwap}. This proves the lemma.
\end{proof}
\begin{figure}[ht]
\includegraphics{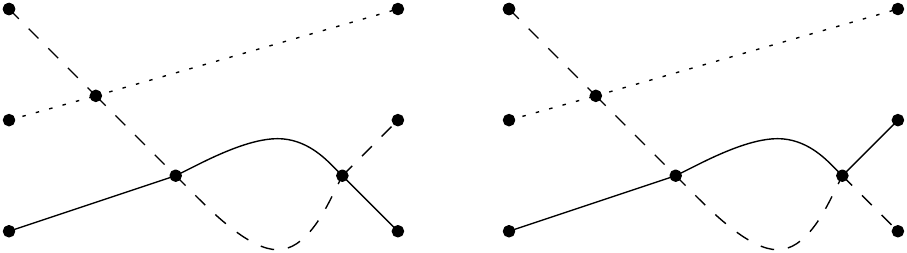}
\caption{An example of swapping a family from $\calP_{-}$ to $\calP_{+}$.  The three paths in the family are given by a dotted, dashed, and solid line, respectively.}
\label{fig:LindstromPathSwap}
\end{figure}

\section{A Fundamental Lemma}

In the proof of the previous lemma we saw that computing minors of the weight matrix involved calculating information involving path families inside the corresponding planar network.  The only path families which survived were those in $\mathcal{P}_{0}$, the vertex-disjoint path families.  Similarly, when considering the minors of minor matrices we will be led to consider families of path families.  To this end we introduce some relevant notation.

Fix a network $\Gamma$ of order $n$.  We will consider two path families living in $\Gamma$.  We think of each of the families as having a different color.  So we let
\begin{align*}
B & =\{\beta_1,\beta_{2},\ldots,\beta_k \} & \text{(colored blue),} \\
R & =\{\rho_1,\rho_2,\ldots,\rho_\ell\}   & \text{(colored red),}
\end{align*}
each be families of paths in $\Gamma$.  We will assume that the paths in $B$ are vertex-disjoint, and similarly the paths in $R$ will be vertex-disjoint.  Sometimes it will be important to emphasize the source and sink of a path.  The notation $\beta_i(a_i,b_i)$ indicates that the path $\beta_i$ begins at source $a_i$ and ends at sink $b_i$.  Thus, our families may be written
\begin{align*}
B& =\{\beta_{1}(a_{1},b_{1}), \beta_{2}(a_{2},b_{2}),\ldots, \beta_{k}(a_{k},b_{k})\}\\
R & =\{\rho_{1}(c_{1},d_{1}),\rho_{2}(c_{2},d_{2}),\ldots, \rho_{\ell}(c_{\ell},d_{\ell})\}.
\end{align*}
We order the paths in $B$ so that the sources of the paths in $B$ follow the natural order in $\Gamma$, and similarly for $R$.  In other words, $a_{1}<a_{2}<\ldots< a_{k}$ and $c_{1}<c_{2}<\ldots<c_{\ell}$.  It may be the case that a pair of paths $\beta_i$ and $\rho_j$ might share several common edges.  The source sets $\{a_1,\ldots,a_k\}$ and $\{c_1,\ldots,c_\ell\}$ are not required to be disjoint from each other, nor are the sink sets $\{b_1,\ldots,b_k\}$ and $\{d_1,\ldots,d_\ell\}$ required to be disjoint from each other.

We construct a certain modified and colored subnetwork of $\Gamma$, which we call $\tilde{\Gamma}$, as follows: First, take the union of the paths in $B$ and $R$ with their respective coloring.  Second, if a single edge of $\Gamma$ is dual-colored we will count this edge with multiplicity two.  Since it is difficult to visualize a dual-colored edge, in pictures we will replace this edge by two edges (without introducing any new intersections), and color the upper edge red and the lower edge blue; this is to enable us to see both colors in figures.  Third, we will slightly perturb our network if necessary so that no intersections, except perhaps at the sources and sinks, occur on the same vertical line.  Fourth, and finally, we remove any vertex which has only a single edge entering the vertex and a single edge exiting that vertex, and we combine those edges into a single edge.  In all subsequent pictures, blue paths will appear with thick lines, while red paths will appear with thin lines.  An example of this process is given in Figure~\ref{fig:ColoredSubGraph}, where our network $\Gamma$ is taken from Figure~\ref{fig:PlanarNetworkExample01}, our families $B$ and $R$ are singleton families involving only one path each. Notice that the last edge of the red path overlaps with the last edge of the blue path, and so we replace that edge with two separate edges (for the simple purpose of visualization).
\begin{figure}[ht]
\includegraphics{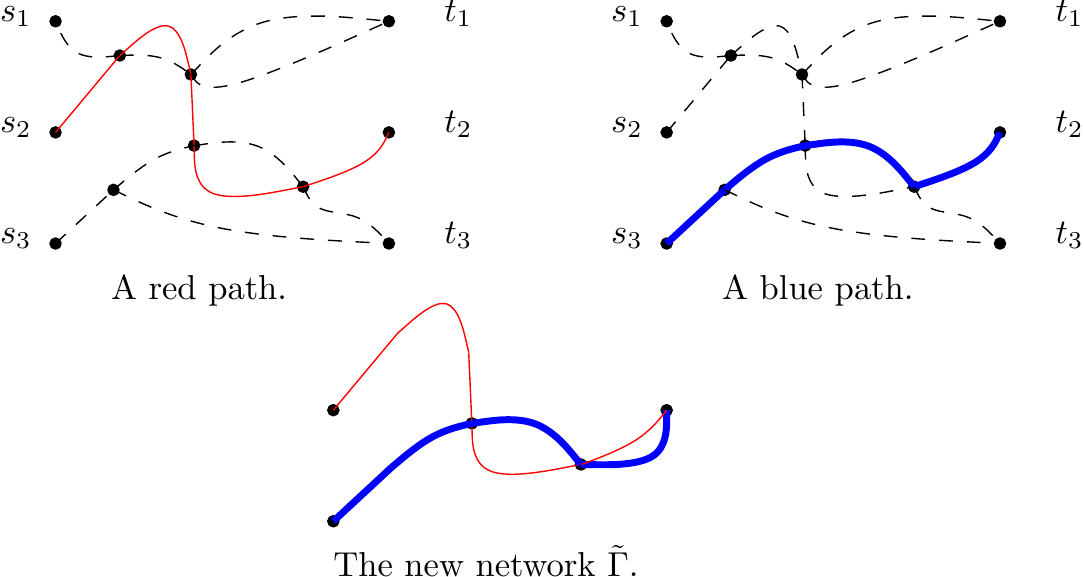}
\caption{Creating the new network $\tilde{\Gamma}$ from $\Gamma$.  The network $\Gamma$ appears with dashed lines.}
\label{fig:ColoredSubGraph}
\end{figure}

In Lemma~\ref{lemma:Lindstrom}, we were able to cancel all terms corresponding to odd permutations by creating a weight-preserving bijection from terms with negative sign to a subset of the paths with positive sign. This bijection is realized geometrically as a path-swap.  Similarly, we will need to swap the sinks of our colored path families.  We want to do so without affecting the sources of the paths, and we want the new families separately to be vertex-disjoint.  In particular, we want $B$ and $R$ to have the same number of paths (say $m$), and we want to be able to recolor edges, in an algorithmic and reversible way to obtain new families
\begin{align*}
B' & = \{\beta'_1(a_1,d_1),\ldots,\beta'_m(a_m,d_m)\}  && \text{(colored blue)}, \\
R' & = \{\rho'_1(c_1,b_1),\ldots,\rho'_m(c_m,b_m)\}    && \text{(colored red)},
\end{align*}
in which the terminal points of the two families of paths have been interchanged, but the set of all edges is the same as the set of edges in the original two families (so as to preserve weights).

Let $e$ be an edge in $\tilde{\Gamma}$, with initial point $s$ and terminal point $t$.  Clearly, if we recolor $e$ we must also recolor any other edge which has initial point $s$, or terminal point $t$; for if not then we will have two paths of the same color entering, or exiting from, a vertex.  With this in mind we make the following definitions:

\begin{definition}
\begin{enumerate}
\item Two different edges are \emph{strongly connected} if they both originate from, or both end in, a common vertex.  Thinking of a dual-colored edge as consisting of two over-lapping edges with different colors, we consider those two edges to be strongly connected to each other.

\item Let $\tilde{\Gamma}$ be a subnetwork of $\Gamma$ formed from the vertex-disjoint path families $B=\{\beta_1,\ldots,\beta_k\}$ and $R=\{\rho_1,\ldots,\rho_{\ell}\}$ as above.  A \textit{chain} in $\tilde{\Gamma}$ is an equivalence class of edges in $\tilde{\Gamma}$ under the reflexive and transitive closure of the strongly connected relation.  Figure~\ref{fig:Chains} gives an example of a colored network $\tilde{\Gamma}$ in which each of the edges in a chain are given the same number.
\end{enumerate}
\end{definition}
\begin{figure}[ht]
\includegraphics{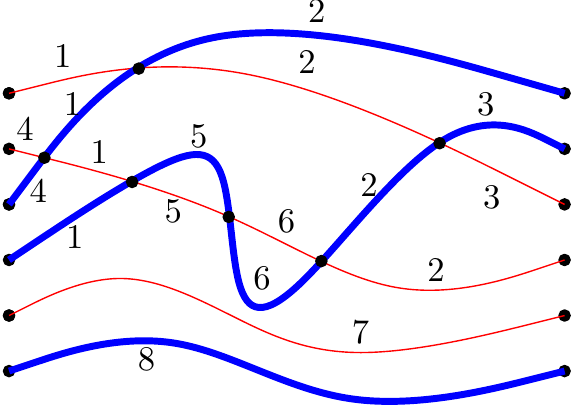}
\caption{Chains in a colored network $\tilde{\Gamma}$.}
\label{fig:Chains}
\end{figure}

Notice that you can travel along a chain by reversing direction and color every time you hit a vertex.

\begin{lemma} \label{lemma:reversechaincolor}
Let $\tilde{\Gamma}$ be the subnetwork of $\Gamma$ formed from (separately) vertex-disjoint path families $B=\{\beta_1,\ldots,\beta_k\}$ and $R=\{\rho_1,\ldots,\rho_{\ell}\}$, where $B$ is colored blue and $R$ is colored red.  Reversing the coloring of all edges in a chain of $\tilde{\Gamma}$ results in a colored network $\tilde{\Gamma}'$  which is the union of a vertex-disjoint blue path family and a vertex-disjoint red path family.
\end{lemma}

Before proving Lemma~\ref{lemma:reversechaincolor}, we caution the reader that, while recoloring a chain of $\tilde{\Gamma}$ preserves the vertex-disjointness property of each colored path families, in general it does not preserve the \textit{number} of blue paths or red paths, as illustrated by the following example.

\begin{example} \label{example:recoloringcounterexample}
In Figure~\ref{fig:Chains}, recoloring the chain numbered 8 increases the number of red paths, while decreasing the number of blue paths.  We leave it to the reader to show that such a recoloring results in a change in the number of paths of a certain color if and only if the chain being recolored has one endpoint which is a source, and another endpoint which is a sink.
\end{example}

\begin{proof}[Proof of Lemma~\ref{lemma:reversechaincolor}]
Let $v$ be a vertex of $\tilde{\Gamma}$ that is not a source nor a sink. Since the red and blue families are vertex-disjoint, to the left of $v$ there are two edges, one blue and one red, or there is a single bi-colored edge. Either way, the edges to the left of $v$ belong to the same chain. If the coloring of the edges of that chain is reversed, there continue to be one blue edge and one red edge or a single bi-colored edge. In other words, reversing the coloring of the chain preserves the number of red and blue edges that meet the vertex $v$ on the left. A similar argument applies to the edges that meet the vertex $v$ on the right. Thus recoloring the chain preserves the fact that there is a red path passing through $v$ and also a blue path passing through $v$.

Similarly, if $v$ is a source or sink attached to a blue edge and a red edge or a single bi-colored edge, recoloring the chain containing those edges preserves the number of blue and red edges touching $v$. So any source point of the new network touches at most one red edge and one blue edge.

It follows that the new network $\tilde{\Gamma}'$ obtained by reversing the colors of a chain is the union of a blue vertex-disjoint path family and a red vertex-disjoint path family.
\end{proof}

Now we need to introduce conditions on the graph which will guarantee that recoloring preserves the number of paths of any given color.

\begin{definition}
We will say that a chain in a graph is \emph{even} if it contains an even number of edges (counting multiplicity), otherwise it is \emph{odd}. A colored network is \emph{evenly chained} if every chain is even.  We say that a chain is a \emph{closed tour} if we can well-order the edges in the chain so that the $i$th edge is strongly connected to the $(i+1)$st edge, and the last edge is strongly connected to the first edge.  Note that a dual-colored edge is a closed tour.  Also, as is evidenced in Figure~\ref{fig:ClosedTour}, vertices can repeat as one performs the tour around such a chain.
\end{definition}

\begin{lemma}\label{lemma:ChainFacts}
Let $\tilde{\Gamma}$ be the subnetwork of $\Gamma$ formed as the union of vertex-disjoint path families $B=\{\beta_1,\ldots,\beta_k\}$ and
$R=\{\rho_1,\ldots,\rho_{\ell}\}$, where $B$ is colored blue and $R$ is colored red. Then
\begin{enumerate}
\item
Any even chain contains the same number of red source points as blue source points (counting multiplicities). Similarly, any even chain contains the same number of blue sink points as red sink points (counting multiplicities).

\item
Any closed tour is even.

\item
An even chain that is not a closed tour has endpoints of opposite color and these endpoints are both sources or both sinks.

\item
Any odd chain contains an odd number of source points (counting multiplicities) and also an odd number of sink points (counting multiplicities).  Since an odd chain is not a closed tour it has two endpoints. One endpoint is a source while the other is a sink. Both endpoints have the same color.
\end{enumerate}
\end{lemma}
\begin{proof}
On any given chain think of the different colored edges as having opposite directions.  (This new direction is merely a tool, and is not to be confused with the fact that our network is directed from the left to the right).  As one travels from one edge in a chain that is strongly connected to another, one must reverse direction.  We can measure the parity in a chain by the number of direction changes.

Closed tours are even because if you leave a vertex $v$ in one direction, you end the chain by coming back to $v$ (on the same side) in the opposite direction.  Any chain which is not a closed tour has endpoints, which must be sources or sinks, since in the formation of $\tilde{\Gamma}$ we removed any vertices (except the sources and sinks) which had only one edge entering and exiting.

The rest of the lemma involves only simple statements about parity and direction.  It may be helpful to note that any source or sink in a chain which is not an endpoint of the chain is both a red and blue vertex, and thus counts an even number of times.
\end{proof}

\begin{figure}[ht]
\includegraphics{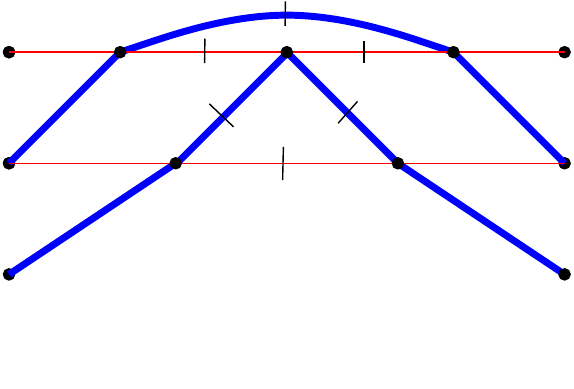}
\caption{An evenly chained colored network, with a single closed tour marked with tick marks.  Notice that as one travels around the chain, there are two loops formed in the underlying graph.}
\label{fig:ClosedTour}
\end{figure}

\begin{lemma}[Fundamental Lemma]  \label{lemma:fundamental}
Let $\tilde{\Gamma}$ be the subnetwork of $\Gamma$ formed from vertex-disjoint path families
\begin{align*}
B& =\{\beta_1(a_1,b_1),\ldots,\beta_m(a_m,b_m)\},  \\
R&=\{\rho_1(c_1,d_1),\ldots,\rho_m(c_m,d_m)\},
\end{align*}
and suppose $B$ and $R$ are evenly chained.
\begin{enumerate}
\item
If a chain in $\tilde{\Gamma}$ contains a source vertex, then the chain contains the same number of source vertices for red paths as it does for blue paths. Similarly, if the chain contains a sink vertex, then the chain contains the same number of sink vertices for red paths as it does for blue paths.
\item
Reversing the colorings of all of the edges in a chain of $\tilde{\Gamma}$ results in an evenly chained network $\tilde{\Gamma}'$ of blue path families and red path families.
\item
There is a unique way to recolor some final edges in both path families involving a minimal number of recoloring of edges in $\tilde{\Gamma}$ that results in a vertex-disjoint, evenly chained family of the form
\begin{align*}
B'& =\{\beta'_1(a_1,d_1),\ldots,\beta'_m(a_m,d_m)\},  \\
R'&=\{\rho'_1(c_1,b_1),\ldots,\rho'_m(c_m,b_m)\},
\end{align*}
in which the sink points of the blue and red families have been interchanged.
\end{enumerate}
\end{lemma}
\begin{proof}
(1) As the families are evenly chained, all chains are even.  Thus, the claims about the number of sources and sinks follow from the previous lemma.

(2) Recoloring all edges in a chain does not change the number of edges in the chain, so the new network $\tilde{\Gamma}'$ is still evenly chained.

(3) When any (final) edge $e$ is recolored, then every other edge in the chain containing $e$ must be recolored, if we are going to preserve vertex-disjointedness.  Furthermore, to swap sinks in our colored families we must at least recolor any edge connected to a sink where that edge is the only one attached to the sink.  Thus, recoloring all chains containing edges attached to sinks, where the sink has only one edge attached, is necessary.  We now show that this is sufficient.

By the previous lemma, such a recoloring will not change the coloring of any source points (although it might interchange the colors of two paths both coming into the same source point).  By applying parts (1) and (2) finitely many times, we see that the resulting colored families will still be evenly chained, with the same number of paths in each family.  By construction, we have reversed the endpoints.  Further, from the fact that each family is still (separately) vertex-disjoint by Lemma~\ref{lemma:reversechaincolor}, and our graph is a subgraph of a planar graph of order $n$, the source points (in their original order) match the sinks in the manner specified in the statement of the lemma.
\end{proof}

If one wants to recolor \emph{all} final edges in both path families, and the associated chains, this also results in a new set of path families with the same properties as in item (3) above.  The only difference between these choices is whether or not one wants to recolor closed tours involving a sink.  (In the example below, such a closed tour is not recolored.  But it could be, if desired.)

\begin{example}
In Figure \ref{fig:MinimalRecolor} there are two graphs which are obtained from one another by a minimal recoloring of edges to preserve disjointness in a given family, but also swaps sinks between the families.

\begin{figure}[ht]
\includegraphics{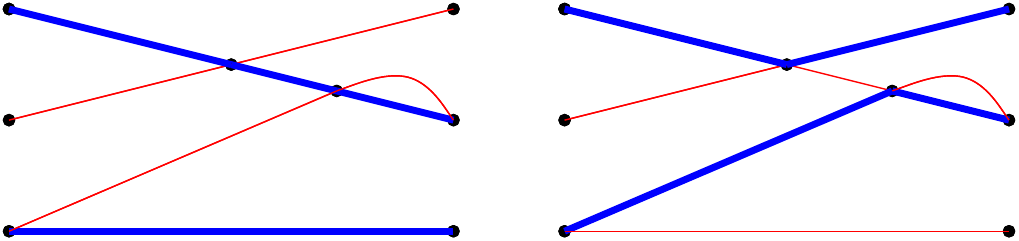}
\caption{A minimal recoloring to swap sinks.}
\label{fig:MinimalRecolor}
\end{figure}
\end{example}

Notice that if we try to recolor the sinks of families which are not evenly chained, we will necessarily have to change the color of some source point, by Lemma~\ref{lemma:ChainFacts} part (4).

We need one more graph theoretic result, which tells us that certain colored networks are necessarily evenly chained.  These graphs will correspond to the entries in a determinant attached to an odd permutation.

\begin{lemma}\label{lemma:crossingareeven}
Let $\tilde{\Gamma}$ be the subnetwork of $\Gamma$ formed from vertex-disjoint path families
\begin{align*}
B& =\{\beta_1(a_1,b_1),\ldots,\beta_m(a_m,b_m)\},  \\
R&=\{\rho_1(c_1,d_1),\ldots,\rho_m(c_m,d_m)\}.
\end{align*}
Further suppose that $a_{i}<c_{i}$ but $b_{i}>d_{i}$ for each $1\leq i\leq m$.  In other words, the path $\beta_{i}$ starts above the path $\rho_{i}$, but ends below it.  Then $\tilde{\Gamma}$ is evenly chained.
\end{lemma}
\begin{proof}
We introduce an auxiliary measure to each edge of a colored graph, which we call the \emph{depth} of an edge $e$, defined by
\[
{\rm depth}(e)=
\begin{cases}
i-k-1 & \begin{array}{l} \text{if $e$ belongs to the $i$th path in $B$}\\ \text{  and $k$ is the number of red paths strictly above $e$,}\end{array}\\
-i+k & \begin{array}{l} \text{if $e$ belongs to the $i$th path in $R$}\\ \text{  and $k$ is the number of blue paths on or above $e$.}\end{array}
\end{cases}
\]
By direct computation one finds that two edges which are strongly connected have the same depth.  Thus depth is an invariant of chains.  The assumptions of the lemma guarantee that all paths start with non-negative depth, but end with negative depth.  Thus no chain contains both a source and a sink, and hence the graph is evenly chained.
\end{proof}

One can view depth as a measure of how much one must perturb a graph where the blue and red paths alternate (with no intersections) to reach the given graph.

\section{Completion of the proof of the Main Theorem} \label{sec:CompletionOfProof}

\begin{proof}[Proof of Theorem~\ref{thm:MainTheorem}]
Let $\Gamma$ be a planar network of order $n$ with weighted edges, and let $W$ be the weight matrix of $\Gamma$.  Let $k\geq 1$, and let $A$ and $B$ be two subsets of $\{0,1,\ldots, n-1\}$ of cardinality $k$.  Let $T=(t_{i,j})$ be the $(A,B)$-minor matrix of $W$.

Lindstr\"{o}m's lemma tells us that $t_{i,j}$ is the sum of the weights of all vertex-disjoint path families from the sources $i+A$ to the sinks $j+B$, through the network $\Gamma$.  Let $\mathcal{P}_{0,i,j}$ denote the set of all such families. Let $C=\{c_{1},c_{2}\}$ be two indices of rows in $T$ with $c_{1}<c_{2}$, and let $D=\{d_{1},d_{2}\}$ be two indices of columns with $d_{1}<d_{2}$.  We want to show that $\det T[C,D]=t_{c_{1},d_{1}}t_{c_{2},d_{2}}-t_{c_{1},d_{2}}t_{c_{2},d_{1}}$ is a subtraction-free expression in the weights of $\Gamma$.

We view any element of $\mathcal{P}_{0,c_{1},\ast}$ as a blue path family, and elements of $\mathcal{P}_{0,c_{2},\ast}$ are red families.  A term in $\det T[C,D]$ involves a subtraction if and only if it corresponds to a term in $t_{c_{1},d_{2}}t_{c_{2},d_{1}}$.  Writing this product as a sum over weights over paths, a single term looks like $w(\boldsymbol{\pi}_{1})w(\boldsymbol{\pi}_{2})$ where $\boldsymbol{\pi}_{1}$ is a blue path family with sources $c_{1}+A$ and sinks $d_{2}+B$, and where $\boldsymbol{\pi}_{2}$ is a red path family with sources $c_{2}+A$ and sinks $d_{1}+B$.  By Lemma~\ref{lemma:crossingareeven}, this is an evenly chained colored network.  By our fundamental lemma, we may recolor edges to swap the sinks, in a unique and reversible way, and get another evenly chained colored network.  This network corresponds to a term in $t_{c_{1},d_{1}}t_{c_{2},d_{2}}$, and thus cancels our original term.
\end{proof}

In examining this proof, one might ask where it breaks down if we try to consider minors of the minor matrix of larger size.  Taking the determinant of a $3\times 3$ submatrix (for example) would correspond to a system of 3-colored vertex-disjoint path families.  Terms with negative sign would still correspond to one of the colored families ``crossing over'' another of the families; and so we can still swap sinks.  But it turns out that this interchanging is not a bijective action in that case.  Two different pairs of families might both be switchable.

\begin{example}\label{example:Order6CounterExample}
Consider Figure~\ref{fig:CounterExample} below.
\begin{figure}[ht]
\includegraphics{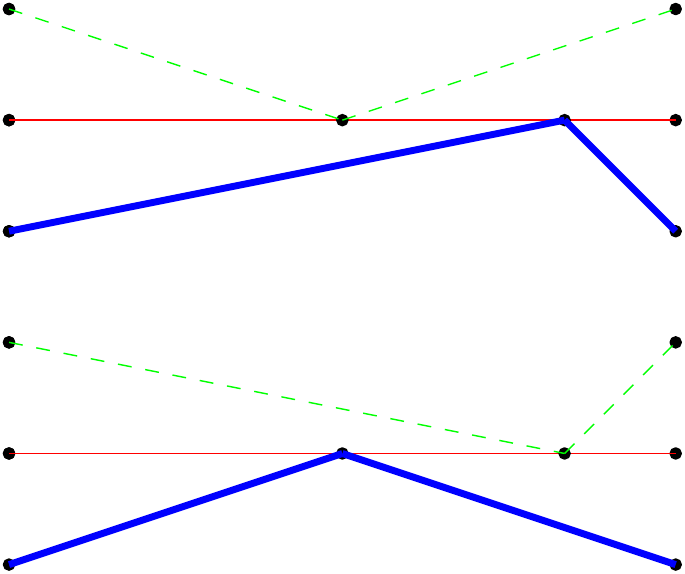}
\caption{A single planar network of order six, colored with 3 colors.  The colors are also represented by the thin, thick, and dashed lines, respectively.}
\label{fig:CounterExample}
\end{figure}

There are exactly three ways to recolor this network, without switching the color of any initial edge.  First, one can leave the diagram alone, and the coloring corresponds to the identity permutation between sources and sinks.  Second, one can switch the end edges of the green and red families (the dashed and thin lines) and obtain an odd permutation.  Third, one could instead switch the end edges of the red and blue families (the thin and thick lines) and also obtain an odd permutation.  There are more odd permutations than even ones.

If we give each edge weight 1, then the weight matrix is
\[
W=\begin{pmatrix}
1 & 1 & 1 & 0 & 0 & 0\\
1 & 1 & 1 & 0 & 0 & 0\\
0 & 1 & 1 & 0 & 0 & 0\\
0 & 0 & 0 & 1 & 1 & 0\\
0 & 0 & 0 & 1 & 1 & 1\\
0 & 0 & 0 & 1 & 1 & 1\\
\end{pmatrix}.
\]
The colors in the diagram imply that we should take $A=B=\{0,3\}$ and form the corresponding $(A,B)$-minor matrix.  If we do so, we obtain
\[
T=\begin{pmatrix}
1 & 1 & 0\\
1 & 1 & 1\\
0 & 1 & 1
\end{pmatrix}.
\]
One computes $\det(T)=-1$, which is the number of even permutations minus the number of odd permutations.

We note in passing that if arbitrary weights are given to the edges in the underlying network of order 6 then every minor of every minor matrix of the weight matrix is a subtraction-free expression in those weights, except for the determinant of the corresponding $3\times 3$ minor matrix that we constructed above.
\end{example}

\section{Open Problems}\label{Section:Open}

While Example~\ref{example:Order6CounterExample} tells us that arbitrary minor matrices of a totally nonnegative matrix no longer have to be totally nonnegative, the coloring on the graph is peculiar, in that the path families are interlaced.  We would like to thank Kelli Talaska for bringing to our attention the following example which shows that similar properties hold even when the families come from simple minor matrices.

\begin{example}
Let $\mathcal{L}$ be the operator on a matrix which gives the $(\{0,1\}\times\{0,1\})$-minor matrix.  Note that the colored path families in $\mathcal{L}(W)$ will consist of two paths whose sources are consecutive.

On page 14 of \cite{MS}, after Conjecture 7.4, is a TN matrix $A$ with $\mathcal{L}(A)$ not TN.  This is an example taken from \cite{FHMJ}.  One can construct a planar network which gives rise to $A$; and Figure~\ref{fig:IteratedNotTN} gives a simplified network with these same properties.

\begin{figure}[ht]
\includegraphics[scale=.5]{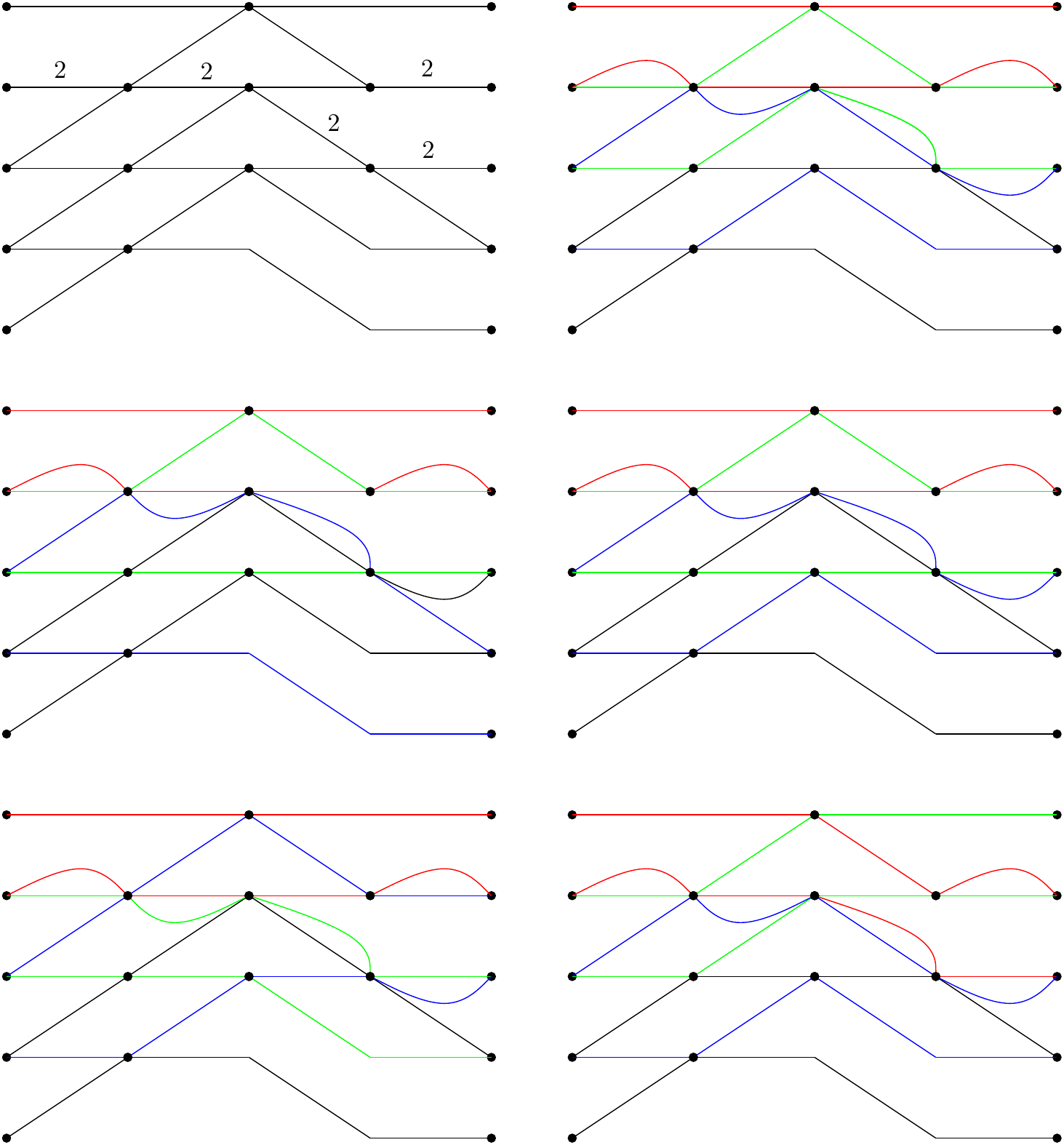}
\caption{A network which can be colored so that the even permutations are outnumbered by the odd permutations.  The $2$'s are on segments which will be double colored (represented by the curvy paths in the later diagrams).  Sources and sinks of any color are consecutive.}
\label{fig:IteratedNotTN}
\end{figure}

If arbitrary weights are assigned to each edge in the network, and $W$ is the corresponding weight matrix, then one can show that each of the minors of $\mathcal{L}(W)$ is subtraction-free in terms of the weights, except the determinant of $\mathcal{L}(W)$.  Furthermore, the same is true if we iterate the $\mathcal{L}$-operator.  So, while the main result of this paper implies that if $A$ is a TN matrix then $\mathcal{L}^{2}(A)$ is nonnegative, this example shows that $\mathcal{L}^{4}(A)$ can be negative.
\end{example}

The infinite log-concavity conjecture is equivalent to showing that the $(A,B)$-minor matrix of $W$, where $A=\{0,1\}$ and $B=\{0,1\}$, is a TN matrix when $W$ arises from a very special network related to a real polynomial with only negative roots.  A prototypical example of such a network is given in Figure~\ref{fig:sequence-network}.  The matrix $W$ will be a Toeplitz matrix.  Intuitively, the infinite log-concavity conjecture should have a purely combinatorial proof which looks at subtraction-free expressions, rather than only an analytic proof relying on properties of the real numbers.

\begin{figure}[ht]
\includegraphics[scale=1]{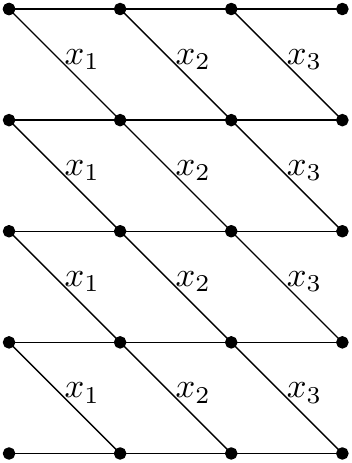}
\caption{A network with three ``columns'' arising from a polynomial with three roots.  More columns can be added if we consider polynomials with more than three roots.}
\label{fig:sequence-network}
\end{figure}

We pose the following conjecture:

\begin{conjec}
Let $W$ be the weight matrix for a planar network of the general form given in Figure~\ref{fig:sequence-network} (but with an arbitrary number of rows and columns).  If $T$ is a matrix formed from $W$ by iterating minor-matrix constructions then all minors of $T$ are subtraction-free expressions in the weights of the planar network.
\end{conjec}

We were led to pose this conjecture after having verified through symbolic computation that the result holds for a large number of columns and rows, and for many iterations of the minor matrix construction.  The following special case would give a new proof of the infinite log-concavity result.

\begin{question}
If $W$ is the weight matrix for a planar network of order $n$, of the general form given in Figure~\ref{fig:sequence-network}, are the minors of $\mathcal{L}(W)$ subtraction-free expressions in those weights?
\end{question}

\section*{Acknowledgements}

We wish to thank Kelli Talaska for pointing out the example in the previous section, and for other useful and insightful comments on an earlier draft of the paper.  We also thank the referee for providing helpful comments which improved the paper.


\bibliographystyle{amsplain}
\bibliography{TNN_minor_bib}

\end{document}